\documentclass[a4, 10pt, reqno]{article}
\usepackage{amscd, amsmath, amssymb, amsfonts, amsthm}
\usepackage[arrow, matrix]{xy}
\usepackage{graphicx}
\usepackage{color}
\usepackage[utf8]{inputenc}
\usepackage[english]{babel}

\theoremstyle{plain}
\newtheorem{theorem}{Theorem}[section]
\newtheorem{lemma}{Lemma}[section]

\theoremstyle{definition}

\numberwithin{equation}{section}

\begin{document}

\title{ Maximum values of the edge Mostar index in tricyclic graphs}

\author{Fazal Hayat$^{a}$,  Shou-Jun Xu$^{a,b}$\footnote{Corresponding author
 \newline E-mail addresses: fhayatmaths@gmail.com (F. Hayat), shjxu@lzu.edu.cn (S. J. Xu), zhoubo@scnu.edu.cn (B. Zhou)}, Bo Zhou$^{c}$ \\
$^a$School of Mathematics and Statistics,  Gansu Center for Applied Mathematics,\\
 Lanzhou University,  Lanzhou 730000,  P.R. China \\
$^b$School of Mathematics and Statistics,   Minnan Normal University,\\
Zhangzhou, Fujian 363000,  P.R. China\\
$^{c}$School of Mathematical Sciences,  South China Normal University, \\
 Guangzhou 510631, P.R. China}

 \date{}
\maketitle

\begin{abstract}
For a  graph $G$, the edge Mostar index of $G$ is the sum of  $|m_u(e|G)-m_v(e|G)|$ over all edges $e=uv$ of $G$,  where $m_u(e|G)$ denotes the number of edges of $G$ that have a smaller distance in $G$ to $u$ than to $v$, and analogously for $m_v(e|G)$. This paper mainly studies the problem of determining the graphs that maximize the edge Mostar index among tricyclic graphs. To be specific, we determine a sharp upper bound for the edge Mostar index on tricyclic graphs and identify the graphs that attain the bound.
\end{abstract}

\section{Introduction}

Let $G = (V, E)$ be a graph with vertex set $V(G)$ and edge set $E(G)$.  The order and size of $G$ are the cardinality of $V(G)$ and $E(G)$, respectively. The distance between $u$ and $v$ in $G$ is the least length of the path connecting $u$ and $v$ denoted by $d_G(u,v)$. For a vertex $x$ and edge $e = uv$ of a graph $G$, the distance between $x$ and $e$, denoted by $d_G (x, e)$ , is defined  as $d_G (x, e)= min\{ d_G(x,u), d_G(x,v)\}$.

A topological index is a real number related to graph $G$.  They are widely used for characterizing molecular graphs, establishing relationships between structure and properties of molecules, predicting biological activity of chemical compounds, and making their chemical applications.

Do\v{s}li\'c et al. \cite{DoM} introduced a bond-additive structural invariant as a quantitative refinement of the distance non-balancedness and also a measure of peripherality in graphs, named the Mostar index. For a  graph $G$, the Mostar index is defined as
\[
 Mo(G)=\sum_{e=uv \in E(G)}|n_u(e|G) - n_v(e|G)|.
\]
where $n_u(e|G)$ is the number of vertices of $G$ closer to $u$ than to $v$ and $n_v(e|G)$ is the number of vertices closer to $v$ than to $u$.

The problem of determining which graphs uniquely maximize (resp. minimize) the Mostar index in some classes of graphs has received much attention. For example,
Do\v{s}li\'c et al. \cite{DoM} studied  the Mostar index of trees and unicyclic graphs, and gave a cut method for computing the Mostar index of
benzenoid systems.  Hayat and Zhou \cite{HZ} determined all the $n$-vertex cacti with the largest Mostar index, and  obtained a sharp upper bound for the Mostar index among cacti of order $n$ with $k$ cycles, and characterized the extremal cacti. Hayat and Zhou  ~\cite{HZ1}  identified those trees with minimum and/or maximum Mostar index in the families of trees of order $n$ with fixed parameters like maximum degree, diameter and the number of pendent vertices. Hayat and Xu \cite{HX1} determined the graphs that maximize and minimize the Mostar index among trees of order $n$, and one of the fixed parameters such as, odd vertices, vertices of degree two, and pendent paths of fixed length.  Hayat and Xu \cite{HX} obtained a lower bound on the Mostar index in tricyclic graphs, and characterized the graphs that attain the bound.
Deng and Li \cite{DL} determined those trees with a given degree sequence have a maximum Mostar index.  The Mostar  index among trees with a given number of segment sequence hve been studied in  \cite{DL1}.
For more studies about the Mostar index see ~\cite{AD,DL2, DL3, HLM, LD, Te,  XZT,  XZT2}.

Arockiaraj et al. \cite{ACT}, introduced the edge Mostar index as a quantitative refinement of the distance non-balancedness, also it can measure the peripherality of every edge and consider the contributions of all edges into a global measure of peripherality for a given chemical graph.  The edge Mostar index of $G$ is  defined as
\[
 Mo_e(G)=\sum_{e=uv \in E(G)}\psi_G(uv),
\]
where $\psi_G(uv)=|m_u(e|G) - m_v(e|G)|$, and $m_u(e|G)$ denotes the number of edges of $G$ that have a smaller distance in $G$ to $u$ than to $v$, and analogously for $m_v(e|G)$.

Imran et al. \cite{IAI} studied the edge Mostar index of chemical structures and nanostructures using graph operations.
Liu et al. \cite{LSX}  determined the extremal values of the edge Mostar index among trees and unicyclic graphs and determined the maximum and the second maximum value of the edge Mostar index among cactus graphs with a given number of vertices.  Ghalavand et al. \cite{GAN} determined the minimum values of the edge Mostar index among bicyclic graphs with fixed size, and characterized the corresponding extremal graphs. The edge Mostar index for several classes of cycle-containing graphs was computed in \cite{H}. Recently, Hayat et al. \cite{HXZ} determined the sharp upper bound for the edge Mostar index on bicyclic graphs with a fixed number of edges, and the graphs that achieve the bound are completely characterized.  Ghanbari and Alikhani \cite{GA} computed the Mostar and edge Mostar index for polymers.

To have a full understanding of the relationship between the edge Mostar index and the structural properties of the graphs, in this paper, we consider
the edge Mostar index among tricyclic graphs, more precisely, we obtain the sharp upper bound for the edge Mostar index on tricyclic graphs with a fixed number of edges, and characterize all the graphs that attain the bound.

\begin{figure}
\centering
\includegraphics[height=6cm, width=14cm]{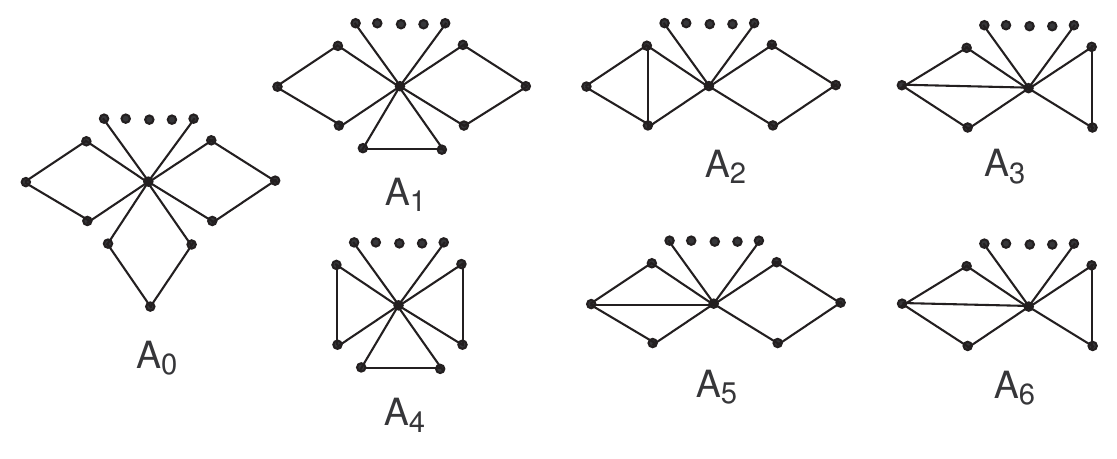}
\caption{ The Graphs $A_i$ $(i= 0,1,...,6)$ of size $m$ in Theorem \ref{a}.}
\label{a}
\end{figure}

\begin{theorem}\label{a} Let $G$ be a tricyclic graph of size $m$. Then
$$Mo_e(G) \leq \left\{\begin{array}{ll}
        12,         & \hbox{if $m=7$, and equality holds iff $G \cong F_1, H_1$}; \\
        23,         & \hbox{if $m=8$, and equality holds iff $G \cong A_3,  F_1, H_1$}; \\
        36,         & \hbox{if $m=9$, and equality holds iff $G \cong  F_1, H_1, A_i (i= 2,...,6)$}; \\
        53,         & \hbox{if $m=10$, and equality holds iff $G \cong A_2$}; \\
        72,         & \hbox{if $m=11$, and equality holds iff $G \cong A_1,  A_2$}; \\
        m^2-m-36,  & \hbox{if $m \geq 12$, and equality holds iff $G \cong A_0$}.
\end{array}\right.$$
(Where $A_i$ $(i= 0,1,...,6)$ are depicted in Fig. \ref{a}, $F_1, H_1$ are depicted and Fig. \ref{f} and Fig. \ref{h}, respectively).
\end{theorem}

In section 2, we give some definitions and preliminary results.  Theorem \ref{a} is proved in section 3.

\section{Preliminaries}

In this section, some basic notations and  elementary results are listed, which will be useful in the proof of main results.

 \begin{figure}
\centering
\includegraphics[height=8cm, width=14cm]{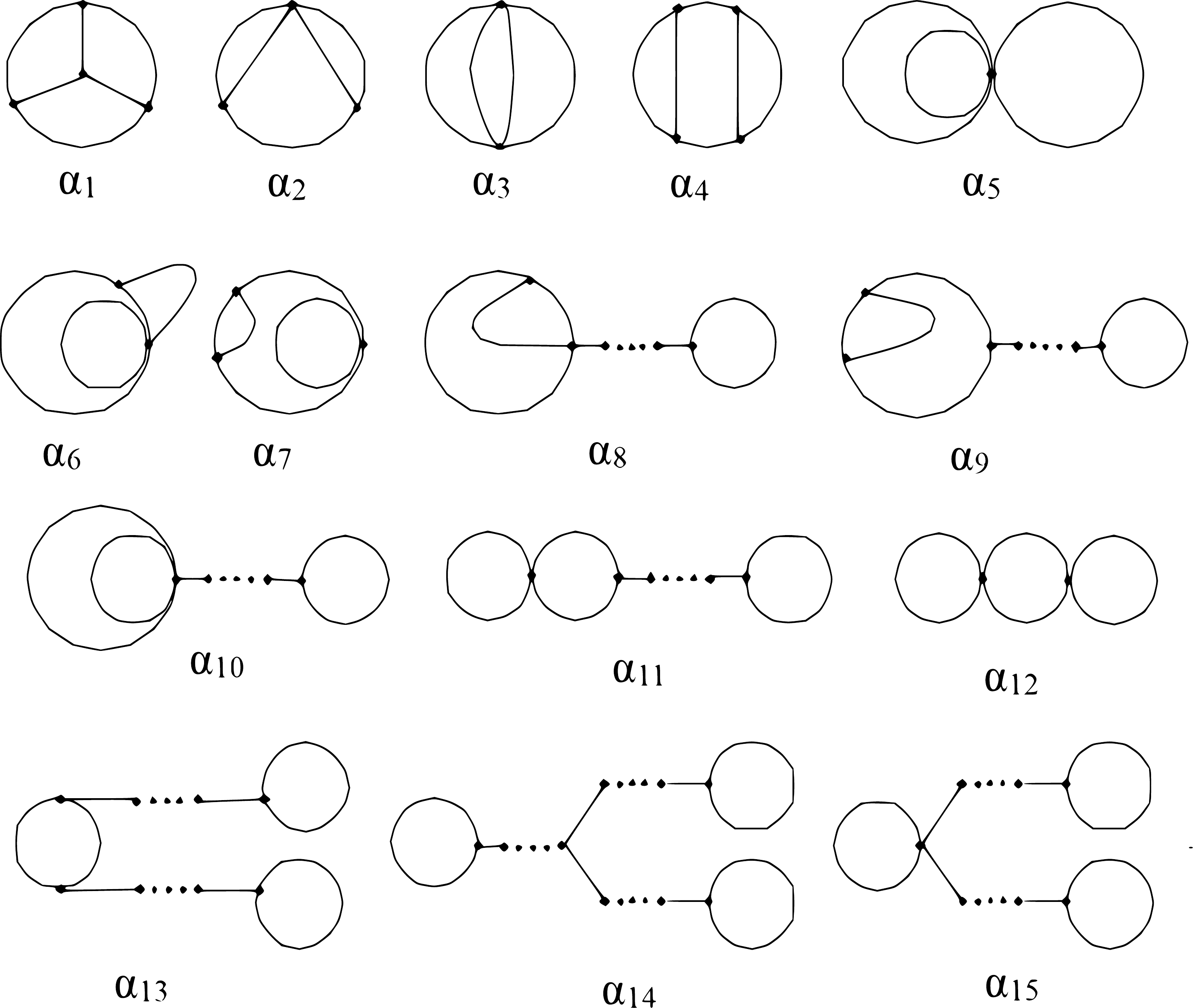}
\caption{ The braces in $ \mathcal{G}_m$.}
\label{alp1}
\end{figure}

For  $v \in V(G)$, let $N_G(v)$  be the set of vertices that are adjacent to $v$ in $G$. The degree of $v \in V(G)$ , denoted by $d_G(v)$,  is the cardinality of  $N_G(v)$. A vertex with degree one is called a pendent vertex and an edge incident to a pendent vertex is called a pendent edge. A graph $G$ with $n$ vertices is a tricyclic graph if $|E(G)|=n+2$. As usual, by $S_n$, $P_n$ and $C_n$ we denote the star, path and cycle on $n$ vertices, respectively.

 Let $G_1 \cdot G_2$ be the graph obtained from $G_1$ and $G_2$ by identifying one vertex, say $u$ of the two graphs.  If $G_1$ contains a cycle and $u$ belongs to some cycle, and $G_2$ is a tree, then we call $G_2$ a pendent tree in $G_1 \cdot G_2$  associated with $u$.  For each $e \in E(G_1)$, every path from $e$ to some edges of $G_2$ passes through $u$. Therefore, the contribution of $G_2$ to  $\sum_{e\in E(G_1)}\psi(e)$ totally depends on the size of $G_2$,  that is, changing the structure of $G_2$ cannot alter the value $\sum_{e\in E(G_1)}\psi(e)$.
 If a graph $H$ is gotten by removing repeatedly all pendents (If any) of $G$. Then we say $H$ is the brace of $G$. That is to say, $H$ does not contain any pendent vertex. Obviously, for all connected tricyclic graphs, their braces are shown in Fig. \ref{alp1}. Let $ \mathcal{G}_m^i$ be the collection whose element includes $\alpha_i$ as its brace for $i=1, \dots , 15$. For convenience, let $ \mathcal{A} = \cup_{i=5}^{15} \mathcal{G}_m^i$.

In the following, Hayat et al. \cite{HXZ} determined a sharp upper bound for the edge Mostar index on bicyclic graphs with a fixed number of edges

\begin{theorem}\label{1} Let $G$ be a bicyclic graph of size $m$. Then
$$Mo_e(G) \leq \left\{\begin{array}{ll}
        4,         & \hbox{if $m=5$, and equality holds iff $G \cong B_3, B_4$}; \\
        m^2-3m-6,  & \hbox{if $6 \leq m \leq 8$, and equality holds iff $G \cong B_1, B_3$};  \\
        48,        & \hbox{if $m=9$, and equality holds iff $G \cong B_0, B_1, B_2, B_3, B_4$};  \\
        m^2-m-24,  & \hbox{if $m \geq 10$, and equality holds iff $G \cong B_0$}.
\end{array}\right.$$
(Where $B_0, B_1, B_2, B_3, B_4$ are depicted in Fig. \ref{b}).
\end{theorem}

Let $S_{m,r} \cong S_{m-r} \cdot C_r$, where the common vertex of $S_{m-r}$ and $C_r$ is the center  of $S_{m-r}$.

\begin{lemma}\cite{HXZ}\label{21} Let $G_1$ be a connected graph of size $m_1$  and  $G_2$ be a unicyclic graph of size $m_2$. Then

$ Mo_e(G_1 \cdot G_2 ) \leq Mo_e(G_1 \cdot S_{m_2, 3} )$ for $m_1 + m_2 \leq 8$;

$  Mo_e(G_1 \cdot G_2 ) \leq  Mo_e(G_1 \cdot S_{m_2, 3} )= Mo_e(G_1 \cdot S_{m_2, 4} )$ for $m_1 + m_2 = 9$;

$ Mo_e(G_1 \cdot G_2 ) \leq  Mo_e(G_1 \cdot S_{m_2, 4} )$ for $m_1 + m_2 \geq 10$;

where the fusing vertex of $G_1 \cdot S_{m_2, 3}$ (resp. $G_1 \cdot S_{m_2, 4}$) is the center of $ S_{m_2, 3}$ (resp. $S_{m_2, 4}$).
\end{lemma}

By means of Theorem \ref{1} and the above result, the following conclusions are obtained.

\begin{lemma}\label{22} Let $G=G_1 \cdot G_2$ be a tricyclic graph, where $G_1$ is a bicyclic graph of size $m_1$ and $G_2$ is a unicyclic graph of size $m_2$. Then

$ Mo_e(G) \leq Mo_e(B_3 \cdot S_{m_2, 3} )$ for  $m_1 + m_2= 8$;

$ Mo_e(G) \leq Mo_e(B_2 \cdot S_{m_2, 3} )= Mo_e(B_3 \cdot S_{m_2, 3} ) = Mo_e(B_3 \cdot S_{m_2, 4})= Mo_e(B_4 \cdot S_{m_2, 3} ) = Mo_e(B_4 \cdot S_{m_2, 4} )$ for  $m_1 + m_2= 9$;

$ Mo_e(G) \leq Mo_e(B_0 \cdot S_{m_2, 4} )$ for  $m_1 + m_2 \geq 12$;

where the fusing vertex of any of the above two graphs is the center of $ S_{m_2, 3}$ or $S_{m_2, 4}$.

\end{lemma}

 \begin{figure}
\centering
\includegraphics[height=2.7cm, width=14cm]{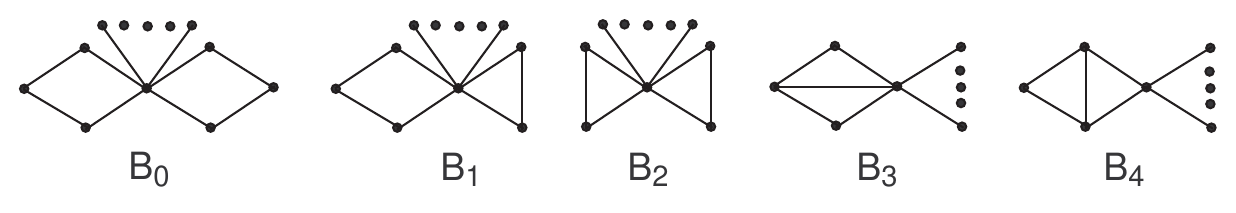}
\caption{The Graphs $B_i$ $(i= 0,1,...,4)$ of size $m$ in Theorem \ref{1}.}
\label{b}
\end{figure}

\section{Proof of Theorem \ref{a}}

For the proof of Theorem \ref{a}, we first develop several lemmas.

First, we obtain  a sharp upper bound for $Mo_e(G)$ on $ \mathcal{A} = \cup_{i=5}^{15} \mathcal{G}_m^i$.

\begin{lemma}\label{31} Let $G \in \mathcal{A}$ of size $m$. Then
$$Mo_e(G) \leq \left\{\begin{array}{ll}
        23,         & \hbox{if $m=8$, and equality holds iff $G \cong A_3$}; \\
        36,         & \hbox{if $m=9$, and equality holds iff $G \cong   A_i (i= 2,...,6)$}; \\
        53,         & \hbox{if $m=10$, and equality holds iff $G \cong A_2$}; \\
        72,         & \hbox{if $m=11$, and equality holds iff $G \cong A_1,  A_2$}; \\
        m^2-m-36,  & \hbox{if $m \geq 12$, and equality holds iff $G \cong A_0$};
\end{array}\right.$$
\end{lemma}
\begin{proof}
Suppose $G \in \mathcal{A}$, then $G$ contains $\alpha_i (i=5,6,...,15)$ as its brace.  Let $G_1$ be a bicyclic graph of size $m_1$  and  $G_2$ be a unicyclic graph of size $m_2$ such that $G = G_1 \cdot G_2 $. Then, in view of Lemmas \ref{21} and \ref{22}, if $m= 8$,  we get
 \begin{eqnarray*}
 Mo_e(G ) & = & Mo_e(G_1 \cdot G_2 ) \leq  Mo_e(G_1 \cdot S_{m_2, 3} ) \\
          &\leq & Mo_e(B_3 \cdot S_{m_2, 3} )= Mo_e(A_3);
\end{eqnarray*}
if $m= 9$,  we get
 \begin{eqnarray*}
 Mo_e(G ) & = & Mo_e(G_1 \cdot G_2 )  \leq Mo_e(B_2 \cdot S_{m_2, 3} )= Mo_e(B_3 \cdot S_{m_2, 3} ) \\
          & = & Mo_e(B_3 \cdot S_{m_2, 4})= Mo_e(B_4 \cdot S_{m_2, 3} ) = Mo_e(B_4 \cdot S_{m_2, 4} )\\
          & = & Mo_e(A_i) (i= 2,...,6);
\end{eqnarray*}
if $m \geq 12$, we have
\begin{eqnarray*}
 Mo_e(G ) & = & Mo_e(G_1 \cdot G_2 ) \leq  Mo_e(G_1 \cdot S_{m_2, 4} )\\
          &\leq & Mo_e(B_0 \cdot S_{m_2, 4} )= Mo_e(A_0).
\end{eqnarray*}
By simple calculation, it is easy to check that,
$Mo_e(A_0)= m^2-m-36$,
$Mo_e(A_1)= Mo_e(A_2)=m^2-2m-27$,
$Mo_e(A_3)= Mo_e(A_4)=m^2-4m-9$,
$Mo_e(A_5)= Mo_e(A_6)=Mo_e(A_7)=m^2-3m-18$.

Clearly,  $Mo_e(A_0)= m^2-m-36 >  Mo_e(A_i) (i=3,...,7)$, for $m \geq 10$, but  $A_0$ contains at least 12 edges. Therefore, if $m=11$, then $Mo_e(A_1)= Mo_e(A_2) >  Mo_e(A_i) (i=3,...,7)$; if $m=10$, then $ Mo_e(A_2) >  A_i (i=3,...,7)$.
\end{proof}

 \begin{figure}
\centering
\includegraphics[height=6cm, width=14cm]{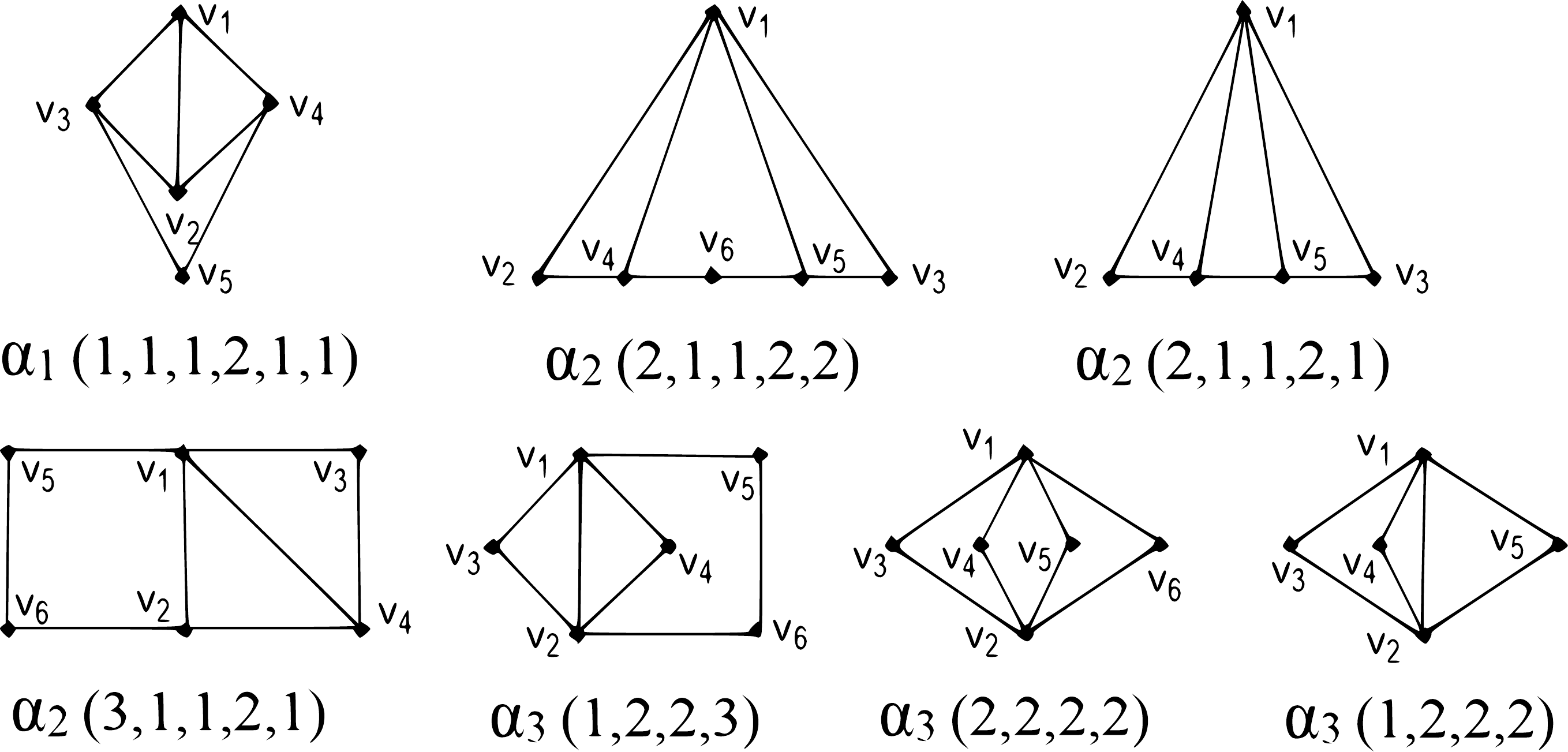}
\caption{ The Graphs for Lemmas \ref{32}, \ref{33}, \ref{34}, \ref{35}, \ref{36}, \ref{37} and \ref{38}.}
\label{alp2}
\end{figure}

In what follows, we determine  a sharp upper bound for $Mo_e(G)$ on $ \cup_{i=1}^{4} \mathcal{G}_m^i$.

For $i=1,...,4$,  $\alpha_i (a_1, a_2, \dots)$  represents the number of edges in each portion of the brace $\alpha_i$.
\begin{lemma}\label{32} Let $G \in \mathcal{G}_m^1$ with brace $\alpha_1 (1,1,1,2,1,1)$. Then
$$Mo_e(G) < \left\{\begin{array}{ll}
      Mo_e(D_1) = m^2-3m-24,  & \hbox{if $7 \leq m \leq 10$};  \\
       Mo_e(D_1) = Mo_e(D_2) = 64,        & \hbox{if $ m=11$};  \\
      Mo_e(D_2) =  m^2-2m-35,  & \hbox{if $m \geq 12$}.
\end{array}\right.$$
\end{lemma}
\begin{proof}
Suppose that $v_i$ $(i=1,...,5) $ be the vertices in $\alpha_1$ of $G$, as shown in Fig. \ref{alp2}.  Let $a_i$ be the number of pendent edges of $v_i$ $(i=1,...,5) $. Suppose that $a_1+a_3 \geq a_2 + a_4 \geq 1$ . Let $G_1$ be the graph obtained from $G$ by shifting $a_2$ (resp. $a_4$) pendent edges from $v_2$ (resp. $v_4$) to $v_1$ (resp. $v_3$). We deduce that
\begin{eqnarray*}
 && Mo_e(G_1 )-  Mo_e(G )=\\
 &&(a_1+a_2-a_3-a_4-a_5)-(a_1+a_4-a_3-a_5)+ (a_3+a_4+a_5+2-3)\\
 &-& (a_2+a_4+3-a_3-a_5-2)+ (a_3+a_4+a_5+2-3)-(a_2+a_4+3-a_3-a_5-2)\\
 & + & (a_1+a_2+a_3+a_4-a_5)-(a_1+a_3-a_4-a_5) + (a_3+a_4+3-a_5-2)\\
 &-&   (a_2+a_3+3-a_4-a_5-2)+ (a_3+a_4+3-a_5-2)-(a_2+a_3+3-a_4-a_5-2)\\
 & + & (a_1+a_2)-(a_1-a_2)+(a_1+a_2+a_3+a_4+3-a_5-1)- (a_1+a_2+a_3+3-a_4-a_5-1)\\
 & + & (a_1+a_2+3-a_3-a_4-a_5-1)-(a_1+a_2+a_4+3-a_3-a_5-1)\\
& = & 2( a_2+a_3+a_4 + a_5 )-2 > 0.
\end{eqnarray*}
For $a_5 >0$, let $G_2$ be the graph obtained from $G_1$ by shifting $a_5$ pendent edges from $v_5$ to $v_3$. We obtain
\begin{eqnarray*}
&& Mo_e(G_2 )-  Mo_e(G_1 )=\\
&& (a_1+a_3+a_5)-(a_1+a_3-a_5)+(a_3+a_5+3-2)-(a_3+3-a_5-2)\\
& + & (a_1+a_3+a_5+3-1)- (a_1+a_3+3-a_5-1)\\
 & = & 6 a_5 > 0.
\end{eqnarray*}
Let $G_3$ be the graph obtained from $G_2$ by shifting $a_1$ pendent edges from $v_1$ to $v_3$. We obtain
\begin{eqnarray*}
 &  & Mo_e(G_3 )-  Mo_e(G_2 )  = \\
 &  &  (a_1+a_3)-(a_3-a_1)+(a_1+a_3+2-3)-(a_3+2-3)+(a_1+a_3+3-2)-(a_3+3-2)\\
 &+ & 0-a_1+(a_1+a_3+1-3)-(a_3+1-a_1-3)\\
 & = & 5 a_1 >0.
\end{eqnarray*}
 Clearly, $G_3 \cong D_2$, and  $G_2 \cong D_1$ for $a_3=0$. Observe that  $Mo_e(D_1 )=m^2-3m-24$, and $Mo_e(D_2 )=m^2-2m-35$ .
\end{proof}

 \begin{figure}
\centering
\includegraphics[height=4cm, width=7cm]{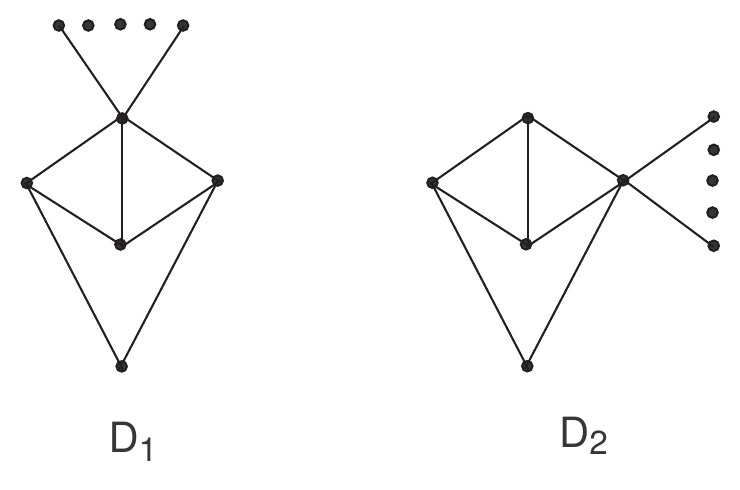}
\caption{ The Graphs $D_1, D_2$ of size $m$ in Lemma \ref{32}.}
\label{d}
\end{figure}

\begin{lemma}\label{41} Let $G \in \mathcal{G}_m^1$ of size $m$. Then
$Mo_e(G) < m^2-m-36$.
\end{lemma}
\begin{proof}
Suppose that $G \in \mathcal{G}_m^1$, then $G$ has a brace $\alpha_1 (a_1, a_2, a_3,  a_4, a_5, a_6)$ as shown in Fig. \ref{alp3}. We consider the following three possible cases.

\noindent {\bf Case 1.}  $\alpha_1$ have at least three paths with length at least two.

\noindent {\bf Subcase 1.1.} The three paths inclose a cycle.

Assume that the three paths are $P(a_1), P(a_2)$ and $P(a_6)$ by the symmetry of $\alpha_1$. We choose  nine edges, two edges in the path $P(a_1)$  such that each one is incident to $x$ or $u$, two edges in the path $P(a_2)$  such that each one is incident to $y$ or $u$, two edges in the path $P(a_6)$  such that each one is incident to $y$ or $z$, one edge in the path $P(a_3)$  incident to $z$, one edge in the path $P(a_4)$  incident to $z$ and one edge in the path $P(a_5)$  incident to $z$. Let $e$ be one of the nine edges. Then $\psi(e) \leq m-7$. This fact is also true for the remaining eight edges. Thus,
\[
Mo_e(G) \leq 9(m-7)+(m-9)(m-1) < m^2-m-36.
\]

\noindent {\bf Subcase 1.2.} The three paths composed a new path.

Assume that the three paths are $P(a_1), P(a_2)$ and $P(a_4)$ by the symmetry of $\alpha_1$. We choose nine edges, two edges in the path $P(a_1)$  such that each one is incident to $x$ or $u$, two edges in the path $P(a_2)$  such that each one is incident to $y$ or $u$, two edges in the path $P(a_4)$  such that each one is incident to $y$ or $z$, one edge in the path $P(a_3)$  incident to $z$, one edge in the path $P(a_5)$  incident to $z$ and one edge in the path $P(a_6)$  incident to $x$.  Thus,
\[
Mo_e(G) \leq 2(m-6)+4(m-7)+2(m-8)+(m-9)+(m-9)(m-1) < m^2-m-36.
\]

\noindent {\bf Subcase 1.3.} The three paths share a common vertex.

Assume that the three paths are $P(a_1), P(a_2)$ and $P(a_3)$ by the symmetry of $\alpha_1$. We choose nine edges, two edges in the path $P(a_1)$  such that each one is incident to $x$ or $u$, two edges in the path $P(a_2)$  such that each one is incident to $y$ or $u$, two edges in the path $P(a_3)$  such that each one is incident to $u$ or $z$, one edge in the path $P(a_4)$  incident to $y$, one edge in the path $P(a_5)$  incident to $z$ and one edge in the path $P(a_6)$  incident to $x$.  We have,
\[
Mo_e(G) \leq 3(m-7)+3(m-8)+3(m-9)+(m-9)(m-1) < m^2-m-36.
\]

\noindent {\bf Case 2.}  $\alpha_1$ have just two paths with length at least two.

\noindent {\bf Subcase 2.1.} The two paths belong to the same cycle at $\alpha_1$.

Assume that the two paths are $P(a_1) $ and $P(a_2)$  by the symmetry of $\alpha_1$. We choose eight edges, two edges in the path $P(a_1)$  such that each one is incident to $x$ or $u$, two edges in the path $P(a_2)$  such that each one is incident to $y$ or $u$, one edge in the path $P(a_3)$  incident to $u$, one edge in the path $P(a_4)$  incident to $y$, one edge in the path $P(a_5)$  incident to $x$ and one edge in the path $P(a_6)$  incident to $x$.  We deduce that,
\[
Mo_e(G) \leq 4(m-6)+3(m-7)+(m-8)+(m-8)(m-1) < m^2-m-36.
\]

\noindent {\bf Subcase 2.2.} The two paths belong to the two different cycles at $\alpha_1$.

We choose eight edges in a similar way, as  in Subcase 2.1. We obtain
\[
Mo_e(G) \leq 4(m-5)+4(m-8)+(m-8)(m-1) < m^2-m-36.
\]

\noindent {\bf Case 3.}  $\alpha_1$ has exactly one path with length at least two.

Assume that the path is $P(a_4) $ with $a_4 \geq 2$. If $a_4=2$, then by Lemma \ref{32}, $Mo_e(G) < m^2-m-36$. If $a_4 \geq 3$, then similarly choose eight edges as  in Subcase 2.1. We obtain
\[
Mo_e(G) \leq 2(m-5)+6(m-8)+(m-8)(m-1) < m^2-m-36.
\]
\end{proof}

\begin{lemma}\label{33} Let $G \in \mathcal{G}_m^2$ with brace $\alpha_2 (2,1,1,2,1)$. Then
$$Mo_e(G) < \left\{\begin{array}{ll}
       Mo_e(F_1) = m^2-4m-9,  & \hbox{if $7 \leq m \leq 16$};  \\
       Mo_e(F_1) = Mo_e(F_2) = 212,        & \hbox{if $ m=17$};  \\
       Mo_e(F_2) = m^2-3m-26,  & \hbox{if $m \geq 18$}.
\end{array}\right.$$
\end{lemma}
\begin{proof}
Suppose that $v_i$ $(i=1,...,5) $ be the vertices in $\alpha_2 (2,1,1,2,1)$ of $G$, as shown in Fig. \ref{alp2}.  Let $a_i$ be the number of pendent edges of $v_i$ $(i=1,...,5) $. Suppose  $a_2+a_4 \geq a_3 + a_5 \geq 1 $. For $a_1 < a_2 +3a_3 + a_5 -3$, let $G_1$ be the graph obtained from $G$ by shifting $a_3$ (resp. $a_5$) pendent edges from $v_3$ (resp. $v_5$) to $v_2$ (resp. $v_4$). We deduce that
\begin{eqnarray*}
& & Mo_e(G_1 )-  Mo_e(G ) = \\
 & & (a_2+a_3+1-a_1-4)-(a_1+a_3+a_5+4-a_2-1)+(1+a_2+a_3-3-a_4-a_5)\\
 & -& (a_2+1-a_4-a_5-3)+(a_1+3-a_4-a_5-2)-(a_1+a_3+3-a_4-2)\\
 &+ & (a_2+a_3+a_4+a_5+3-3)-(a_2+a_4+3-a_5-a_3-3)+(a_1+a_2+a_3+a_4+a_5+4-1)\\
 & -& (a_1+a_2+a_4+4-a_3-1)+(a_4+a_5+3-1)-(a_4+a_5+3-1-a_3)\\
 &+ & (a_1+a_2+a_3+3-2)-(a_1+a_2+a_3-a_5-2)\\
 & = & 2 a_2+6a_3+2a_5 -2a_1-6 > 0.
\end{eqnarray*}
Let $G_2$ be the graph obtained from $G_1$ by shifting $a_4 $ pendent edges from $v_4$ to $v_1$. We obtain
\begin{eqnarray*}
 & & Mo_e(G_2 )-  Mo_e(G_1 )= \\
 & & (a_1+a_4+4-a_2-1)-(a_1+4-a_2-1)+ (a_1+a_4+3-2)-(a_1+3-a_4-2)\\
 &+ & (a_1+a_2+a_4+3-2)-(a_1+a_2+3-2)+(a_2+1-3)-(a_2+1-3-a_4)+(a_2+3-3)\\
 &- & (a_2+a_4+3-3)+(3-1)-(a_4+3-1)\\
 & = & 3 a_4 > 0.
\end{eqnarray*}
 Let $G_3$ be the graph obtained from $G_2$ by shifting $a_2 $ pendent edges from $v_2$ to $v_1$. We obtain
\begin{eqnarray*}
  &  & Mo_e(G_3 )-  Mo_e(G_2 )= \\
 &  &  (a_1+a_2+4-1)-(a_1+4-a_2-1)+(a_1+a_2+3-2)-(a_1+3-2)+(1-3)\\
  & - & (a_2+1-3)+0-(a_2+3-3)\\
& = & 2 a_2 > 0.
\end{eqnarray*}
For $a_1 >6-2a_2$, let $G_4$ be the graph obtained from $G_3$ by shifting $a_1 $ pendent edges from $v_1$ to $v_2$. We have
\begin{eqnarray*}
 & & Mo_e(G_3 )-  Mo_e(G_2 )= \\
  & &(a_1+a_2+1-4)-(a_1+4-a_2-1)+(3-2)-(a_1+3-2)+(a_1+a_2+1-3)\\
                       &\qquad - & (a_2+1-3)+(a_1+a_2+3-3)-(a_2+3-3)\\
                         & = &a_1+ 2 a_2-6 > 0.
\end{eqnarray*}
Clearly, $G_3 \cong F_1$ and  $G_4 \cong F_2$. By simple calculation, we have  $Mo_e(F_1 )=m^2-4m-9$, and $Mo_e(F_2 )=m^2-3m-26$.
\end{proof}

 \begin{figure}
\centering
\includegraphics[height=4cm, width=13cm]{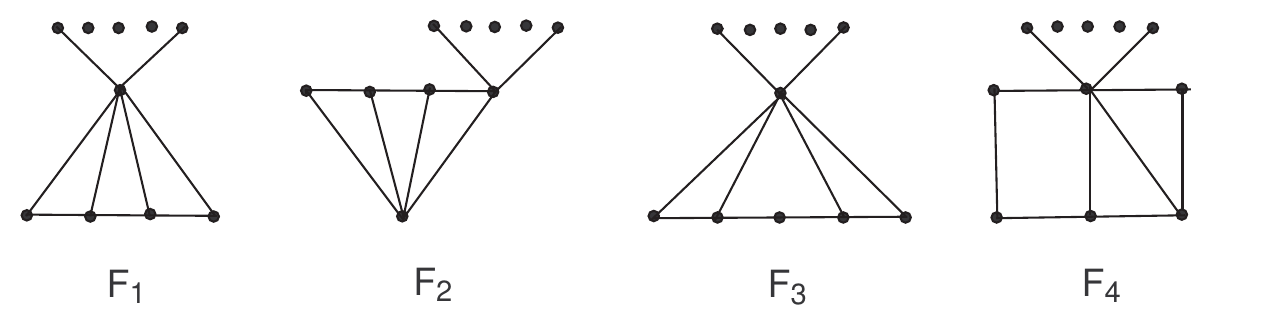}
\caption{ The Graphs $F_1, F_2, F_3, F_4$ of size $m$ in Lemmas \ref{33}, \ref{34} and \ref{35}.}
\label{f}
\end{figure}

\begin{lemma}\label{34} Let $G \in \mathcal{G}_m^2$ with brace $\alpha_2 (2,1,1,2,2)$. Then
$Mo_e(G) < Mo_e(F_3)= m^2-3m-20$.
\end{lemma}
\begin{proof}
Suppose that $v_i$ $(i=1,...,6) $ be the vertices in $\alpha_2 (2,1,1,2,2)$ of $G$, as shown in Fig. \ref{alp2}.  Let $a_i$ be the number of pendent edges of $v_i$ ($i=1,...,6$). For $a_6 >0$, let $G_1$ be the graph obtained from $G$ by shifting $a_6$  pendent edges from $v_6$  to $v_1$. We obtain
\begin{eqnarray*}
 & & Mo_e(G_1 )-  Mo_e(G )=\\
  &  & (a_1+a_3+ a_5+ a_6+5-a_2-1)-(a_1+a_3+a_5+5-a_2-1)\\
 & +& (a_1+a_2+a_4+a_6+5-a_3-1)-(a_1+a_2+a_4+5-a_3-3)\\
 & +& (a_1+a_3+a_5+a_6+4-a_4-2)-(a_1+a_3+a_5+4-a_4-a_6-2)\\
  &+ & (a_1+a_2+a_4+a_6+4-a_5-2)-(a_1+a_2+a_4+4-a_5-a_6-2)\\
 & + & (a_1+a_2+a_4+a_6+4-a_5-2)-(a_1+a_2+a_4+4-a_5-a_6-2)\\
 & + & (a_1+a_3+a_5+4-a_4-2)-(a_1+a_3+a_5+4-a_4-a_6-2)\\
 & + & (a_2+1-a_4-3)-(a_2+1-a_4-a_6-3)+(a_3+1-a_5-3)-(a_3+1-a_5-a_6-3)\\
                       & = & 11a_6 > 0.
\end{eqnarray*}
For $a_2+a_3>a_1$, let $G_2$ be the graph obtained from $G_1$ by shifting $a_3$ (resp. $a_5$) pendent edges from $v_3$ (resp. $v_5$) to $v_2$ (resp. $v_4$). We deduce that
\begin{eqnarray*}
 &  & Mo_e(G_2 )-  Mo_e(G_1 )=\\
  &  & (a_2+a_3+1-a_1-5)-(a_1+a_3+a_5+5-a_2-1)+(a_1+a_2+a_3+a_4+a_5+5-1)\\
  & - & (a_1+a_2+a_4+5-a_3-1)+(a_1+4-a_4-a_5-2)-(a_1+a_3+a_5+4-a_4-2)\\
&+ & (a_1+a_2+a_3+a_4+a_5+4-2)-(a_1+a_2+a_4+4-a_5-2)+(a_2+a_3+1-a_4-a_5-3)\\
&- & (a_2+1-a_4-3)+(3-1)-(a_3+1-a_5-3)+(a_1+a_2+a_3+a_4+a_5-2)\\
& - & (a_1+a_2+a_4+4-a_5-2)+(a_1+4-a_4-a_5-2)-(a_1+a_3+a_5+4-a_4-2)\\
                       & = & 2a_2+2a_3-2a_1 > 0.
\end{eqnarray*}
For $a_2+a_4\geq 1$, let $G_3$ be the graph obtained from $G_2$ by shifting $a_2$ (resp. $a_4$) pendent edges from $v_2$ (resp. $v_4$) to $v_1$ (resp. $v_4$). We have
\begin{eqnarray*}
&  &  Mo_e(G_3 )-  Mo_e(G_2 ) =\\
&  & (a_1+a_2+a_4+5-1)-(a_1+5-a_2-1)+(a_1+a_2+a_4+5-1)-(a_1+5-a_2-1)\\
&+ & (a_1+a_2+a_4+4-2)-(a_1+a_3+a_5+4-a_4-2)+(1-3)-(a_2+1-a_4-3)\\
                       & + & (a_1+a_2+a_4+4-2)-(a_1+4-a_4-2)\\
                       & = & 5a_2+7a_4 > 0.
\end{eqnarray*}
Clearly, $G_3 \cong F_3$, and   $Mo_e(F_3 )=m^2-3m-20$.
\end{proof}

\begin{lemma}\label{35} Let $G \in \mathcal{G}_m^2$ with brace $\alpha_2 (3,1,1,2,1)$. Then
$Mo_e(G) < Mo_e(F_4)= m^2-2m-33$.
\end{lemma}
\begin{proof}
Suppose that $v_i$ $(i=1,...,6) $ be the vertices in $\alpha_2 (3,1,1,2,1)$ of $G$, as shown in Fig. \ref{alp2}.  Let $a_i$ be the number of pendent edges of $v_i$ ($i=1,...,6$). For $a_6 >0$, let $G_1$ be the graph obtained from $G$ by shifting $a_5$ (resp. $a_6$)  pendent edges from $v_5$ (resp. $v_6$)  to $v_1$. We obtain
\begin{eqnarray*}
&  &  Mo_e(G_1 )-  Mo_e(G )=\\
  &  & (a_1+a_3+ a_5+ a_6+3-a_2-2)-(a_1+a_3+a_5+3-a_2-a_6-2)\\
                       &+ & (a_1+a_2+a_3+a_4+a_5+a_6+5-1)-(a_1+a_2+a_3+a_4+5-a_5-a_6-1)\\
                       & + & (a_1+a_2+a_3+a_4+a_5+a_6+5-1)-(a_1+a_2+a_3+a_4+5-a_5-a_6-1)\\
                       &+ & (a_1+a_3+a_5+a_6+3-a_2-2)-(a_1+a_3+a_5+3-a_2-a_6-2)\\
                       &+ & (a_2+a_4+3-a_3-1)-(a_2+a_4+a_6+3-a_3-1)+(a_3+a_4+2-a_2-3)\\
                       &- & (a_2+a_6+3-a_3-a_4-2)+(a_1+a_5+a_6+4-a_2-2)-(a_1+a_5+4-a_4-2)\\
                       & = & 2a_3+4a_5+7a_6-2a_2-2 > 0.
\end{eqnarray*}
 Let $G_2$ be the graph obtained from $G_1$ by shifting $a_3$ (resp. $a_4$) pendent edges from $v_3$ (resp. $v_4$) to $v_1$ (resp. $v_2$). We deduce that
\begin{eqnarray*}
&  & Mo_e(G_2 )-  Mo_e(G_1 ) =\\
 &  & (a_1+a_2+a_3+a_4+2-3)-(a_1+a_3+3-a_2-2)+(a_1+a_2+a_3+a_4+2-3)\\
                       & - & (a_1+a_3+3-a_2-2)+(a_1+a_2+a_3+a_4+5-1)-(a_1+a_2+5-a_3-1)\\
                       &+ & (a_1+a_2+a_3+a_4+3-1)-(a_2+a_4+3-a_3-1)+(a_1+a_2+a_3+a_4+3-2)\\
                       & - & (a_2+3-a_3-a_4-2)\\
                       & = &a_1+ 3a_2+6a_3+5a_4 > 0.
\end{eqnarray*}
 Let $G_3$ be the graph obtained from $G_2$ by shifting $a_1$  pendent edges from $v_1$ to $v_2$. We obtain
\begin{eqnarray*}
&  & Mo_e(G_3 )-  Mo_e(G_2 ) =\\
 &  & (a_1+a_2+2-3)-(a_1+3-a_2-2)+(a_1+a_2+2-3)-(a_2+2-a_1-3)\\
                       & + & (a_1+a_2+3-1)-(a_2+3-1)+(a_1+a_2+3-2)-(a_2+3-2)+(4-2)-(a_1+4-2)\\
                       & = &3a_1+ 2a_2-2 > 0.
\end{eqnarray*}
Thus, $ Mo_e(G )< Mo_e(G_1 )< Mo_e(G_2 )< Mo_e(G_3 )$. Clearly, $G_3 \cong F_4$, and $Mo_e(F_4) =  m^2-2m-33$.
\end{proof}

\begin{lemma}\label{42} Let $G \in \mathcal{G}_m^2$ of size $m$. Then
$Mo_e(G) < m^2-m-36$ for $m \geq 9$, and $Mo_e(G) \leq Mo_e(F_1)$ for $m \leq 9$.
\end{lemma}
\begin{proof}
Suppose that $G \in \mathcal{G}_m^2$, then $G$ has a brace $\alpha_2 (a_1, a_2, a_3,  a_4, a_5)$ as shown in Fig. \ref{alp3}. Assume that $a_4, a_5 \geq 2$. We consider the following three possible cases.

\noindent {\bf Case 1.} $a_4, a_5 \geq 3$.

\noindent {\bf Subcase 1.1.} $a_1= a_2= a_3 =1$.

 We choose nine edges, three edges in the path $P(a_4)$  such that two are incident to $x$  or $y$ and one is in the middle of $P(a_4)$, three edges in the path $P(a_5)$  such that two are incident to $x$ or $z$ and one is in the middle of $P(a_5)$, one edge in the path $P(a_2)$  incident to $x$, one edge in the path $P(a_3)$  incident to $x$ and one edge in the path $P(a_1)$  incident to $y$. We have
\[
Mo_e(G) \leq 4(m-4)+4(m-7)+(m-9)+(m-9)(m-1) < m^2-m-36.
\]

\noindent {\bf Subcase 1.2.} At least one of $a_1, a_2, a_3$ is greater than 1.

If $a_2, a_3 \geq 2$, then we choose 10 edges, three edges in the path $P(a_4)$  such that two are incident to $x$  or $y$ and one is in the middle of $P(a_4)$, three edges in the path $P(a_5)$  such that two are incident to $x$ or $z$ and one is in the middle of $P(a_5)$, two edges  in the path $P(a_2)$  incident to $x$ or $y$, one edge in the path $P(a_3)$  incident to $x$ and one edge in the path $P(a_1)$  incident to $y$. We have
\[
Mo_e(G) \leq 2(m-4)+ (m-5)+2(m-6)+2(m-8)+3(m-9)+(m-10)(m-1) < m^2-m-36.
\]

If $a_1 \geq 2$, then we choose 10 edges, three edges in the path $P(a_4)$  such that two are incident to $x$  or $y$ and one is in the middle of $P(a_4)$, three edges in the path $P(a_5)$  such that two are incident to $x$ or $z$ and one is in the middle of $P(a_5)$, two edges  in the path $P(a_2)$  incident to $x$ or $y$, one edge in the path $P(a_3)$  incident to $x$ and two edges in the path $P(a_1)$  incident to $y$ or $z$. We obtain
\[
Mo_e(G) \leq 4(m-4)+ 6(m-7)+(m-10)(m-1) < m^2-m-36.
\]

\noindent {\bf Case 2.} $a_4 \geq 3, a_5 = 2$.

\noindent {\bf Subcase 2.1.}$a_4 \geq 4, a_5 = 2$, and $a_1= a_2= a_3 =1$.

 We choose nine edges, four edges in the path $P(a_4)$  such that two are incident to $x$  or $y$ and two are in the middle of $P(a_4)$, two edges in the path $P(a_5)$  such that one is incident to $x$  and one is in the middle of $P(a_5)$, one edge in the path $P(a_2)$  incident to $x$, one edge in the path $P(a_3)$  incident to $x$ and one edge in the path $P(a_1)$  incident to $y$. We have
\[
Mo_e(G) \leq (m-4)+2(m-5)+(m-6)+2(m-7)+3(m-8)+(m-9)(m-1) < m^2-m-36.
\]

\noindent {\bf Subcase 2.2.}$a_4 =3, a_5 = 2$, and $a_1= a_2= a_3 =1$.

The Subcase follows from Lemma \ref{35}.

\noindent {\bf Subcase 2.3.}$a_4 \geq 3, a_5 = 2$, and at least one of $a_1, a_2, a_3$ is greater than 1.

The proof is similar to the Subcase 2.1.

\noindent {\bf Case 3.} $a_4 = a_5 = 2$.

\noindent {\bf Subcase 3.1.} At least one of $a_1, a_2, a_3$ is greater than 1.

If $a_2, a_3 \geq 2$, then we choose eight edges, three edges in the path $P(a_4)$  such that two are incident to $x$  or $y$ and one is in the middle of $P(a_4)$, two edges in the path $P(a_5)$  such that one is incident to $x$  and other is in the middle of $P(a_5)$, two edges  in the path $P(a_2)$  incident to $x$ or $y$, one edge in the path $P(a_3)$  incident to $x$ and one edge in the path $P(a_1)$  incident to $y$. We have
\[
Mo_e(G) \leq 4(m-5)+4(m-7)+(m-8)(m-1) < m^2-m-36.
\]

If $a_1 \geq 3$, then we choose nine edges, two edges in the path $P(a_4)$  such that one is incident to $x$  and other is in the middle of $P(a_4)$, two edges in the path $P(a_5)$  such that one is incident to $x$  and the other is in the middle of $P(a_5)$, one edge  in the path $P(a_2)$  incident to $x$, one edge in the path $P(a_3)$  incident to $x$ and three edges in the path $P(a_1)$ such that two are  incident to $y$ or $z$ and one is in the middle of $P(a_1)$. We obtain
\[
Mo_e(G) \leq 2(m-5)+ 2(m-6)+4(m-7)+(m-9)+(m-9)(m-1) < m^2-m-36.
\]

If $a_1 = 2$, then by Lemma \ref{34},
$
Mo_e(G) \leq  m^2-3m-20 < m^2-m-36.
$

\noindent {\bf Subcase 3.2.} $a_1= a_2= a_3 =1$.

By Lemma \ref{33}, we have $Mo_e(G) < m^2-m-36$ for $m \geq 9$, and $Mo_e(G) \leq Mo_e(F_1)$ for $m \leq 9$.
\end{proof}

 \begin{figure}
\centering
\includegraphics[height=4cm, width=13cm]{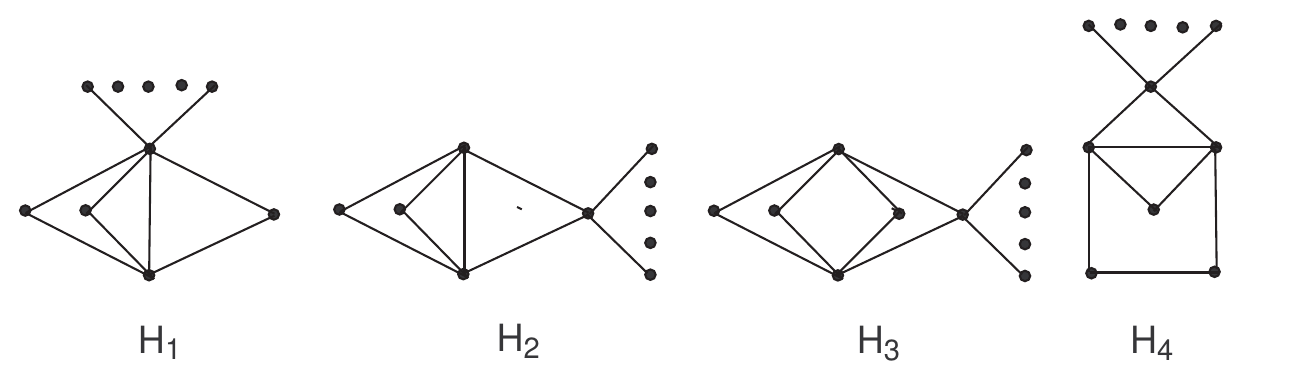}
\caption{ The Graphs $H_1, H_2, H_3, H_4$ of size $m$ in Lemmas \ref{36}, \ref{37} and \ref{38}.}
\label{h}
\end{figure}

\begin{lemma}\label{36} Let $G \in \mathcal{G}_m^3$ with brace $\alpha_3 (1,2,2,2)$. Then
$$Mo_e(G) < \left\{\begin{array}{ll}
   Mo_e(H_1)=     m^2-4m-9,  & \hbox{if $7 \leq m \leq 10$};  \\
    Mo_e(H_1)= Mo_e(H_2)=    68,        & \hbox{if $ m=11$};  \\
    Mo_e(H_2)=    m^2-2m-31,  & \hbox{if $m \geq 12$}.
\end{array}\right.$$
\end{lemma}
\begin{proof}
Suppose that $v_i$ $(i=1,...,5) $ be the vertices in $\alpha_3 (1,2,2,2)$ of $G$ with $d_G(v_1)=d_G(v_2)=4$ and  $d_G(v_3)=d_G(v_4)=d_G(v_5)=2$, as shown in Fig. \ref{alp2}.  Let $a_i$ be the number of pendent edges of $v_i$ $(i=1,...,5) $. Suppose that $a_3 \geq  a_4 \geq a_5$. For $a_4 +a_5 > a_1+a_2+8$, let $G_1$ be the graph obtained from $G$ by shifting $a_4$ (resp. $a_5$) pendent edges from $v_4$ (resp. $v_5$) to $v_3$. We deduce that
\begin{eqnarray*}
&  &  Mo_e(G_1 )-  Mo_e(G ) =\\
 & & (a_3+a_4+a_5+1-a_1-3)-(a_1+a_4+a_5+3-a_3-1)+(a_3+a_4+a_5+1-a_2-3)\\
                       &- & (a_1+a_4+a_5+3-a_3-1)+(a_1+a_3+a_4+a_5+3-1)-(a_1+a_3+a_5+3-a_4-1)\\
                       &+ & (a_2+a_3+a_4+a_5+3-1)-(a_2+a_3+a_5+3-a_4-1)+ (a_1+a_3+a_4+a_5+3-1)\\
                       &- &(a_1+a_3+a_4+3-a_5-1)+(a_1+a_3+a_4+a_5+3-1)-(a_2+a_3+a_4+3-a_5-1)\\
                       & = & 4( a_3+a_4 + a_5 )-2(a_1+a_2) -8> 0.
\end{eqnarray*}
For $a_2 +a_3 > 1$, let $G_2$ be the graph obtained from $G_1$ by shifting $a_2$ ) pendent edges from $v_2$  to $v_1$. We have
\begin{eqnarray*}
 &  &Mo_e(G_2 )-  Mo_e(G_1 ) =\\
 &  & (a_1+a_2+3-a_3-1)-(a_1+3-a_3-1)+(a_3+1-3)-(a_2+3-a_3-1)\\
                       & + & (a_1+a_2+3-3)-(a_1+3-a_2-3)+(a_1+a_2+a_3+3-1)-(a_1+a_3+3-1)\\
                       &+ & (a_3+3-1)-(a_2+a_3+3-1)+(a_1+a_2+a_3+3-1)-(a_1+a_3+3-1)\\
                       &+ & (a_3+3-1)-(a_2+a_3+3-1)\\
                       & = & 2( a_2+a_3 )-4> 0.
\end{eqnarray*}
Clearly, $G_2 \cong H_2$ for $a_1=0, a_3 >0$, and $G_2 \cong H_1$ for $a_3=0, a_1 >0$. For $a_1 +a_3 > 2$, let $G_3$ be the graph obtained from $G_2$ by shifting $a_1$  pendent edges from $v_1$  to $v_3$. We have
\begin{eqnarray*}
&  & Mo_e(G_2 )-  Mo_e(G_1 ) =\\
 &  & (a_1+a_3+1-3)-(a_1+3-a_3-1)+(a_1+a_3+1-3)-(a_3+1-3)+(3-3)\\
                       & - & (a_1+3-3)+(a_1+a_3+3-1)-(a_3+3-1)+(a_1+a_3+3-1)-(a_3+3-1)\\
                       & = & 2( a_1+a_3 )-4> 0.
\end{eqnarray*}
Thus, $ Mo_e(G )< Mo_e(G_1 )< Mo_e(G_2 )< Mo_e(G_3 )$. Clearly, $G_3 \cong H_2$, and by simple calculation, we deduce that $Mo_e(H_2) =  m^2-2m-31$, $Mo_e(H_1) =  m^2-4m-9$.
\end{proof}

\begin{lemma}\label{37} Let $G \in \mathcal{G}_m^3$ with brace $\alpha_3 (2,2,2,2)$. Then
$Mo_e(G) < Mo_e(H_3) = m^2-m-48$.
\end{lemma}
\begin{proof}
Suppose that $v_i$ $(i=1,...,6) $ be the six vertices in $\alpha_3 (2,2,2,2)$ of $G$ with $d_G(v_1)=d_G(v_2)=4$ and  $d_G(v_3)=d_G(v_4)=d_G(v_5)=d_G(v_6)=2$, as shown in Fig. \ref{alp2}.  Let $a_i$ be the number of pendent edges of $v_i$ $(i=1,...,6) $. Suppose that $a_3 \geq  a_4 \geq a_5 \geq a_6>0$. Let $G_1$ be the graph obtained from $G$ by shifting $a_i$ ( $i \geq 4$) pendent edges from $v_i$ ( $i \geq 4$) to $v_3$. We obtain
\begin{eqnarray*}
 &  & Mo_e(G_1 )-  Mo_e(G )=\\
  & & (a_2+a_3+ a_4+a_5+ a_6+1-a_1-3)-(a_1+ a_4+a_5+a_6+3-a_2-a_3-1)\\
                       & + & (a_1+a_3+a_4+a_5+a_6+1-a_2-3)-(a_2+a_4+a_5+a_6+3-a_1-a_3-1)\\
                       &+ & (a_1+a_3+a_4+a_5+a_6+3-a_2-1)-(a_1+a_3+a_5+a_6+3-a_4-a_2-1)\\
                       & + &(a_2+a_3+a_4+a_5+a_6+3-a_1-1)-(a_2+a_3+a_5+a_6+3-a_1-a_4-1)\\
                       & + & (a_1+a_3+a_4+a_5+a_6+3-a_2-1)-(a_1+a_3+a_4+a_6+3-a_2-a_5-1)\\
                       &+ & (a_2+a_3+a_4+a_5+a_6+3-a_1-1)-(a_2+a_3+a_4+a_6+3-a_1-a_5-1)\\
                       &+ & (a_1+a_3+a_4+a_5+a_6+3-a_2-1)-(a_1+a_3+a_4+a_5+3-a_2-a_6-1)\\
                       &+ & (a_2+a_3+a_4+a_5+a_6+3-a_1-1)-(a_2+a_3+a_4+a_5+3-a_1-a_6-1)\\
                       & = & 4(a_3+a_4+a_5+a_6)-8 > 0.
\end{eqnarray*}
For $a_1 +a_2 > 0$, let $G_2$ be the graph obtained from $G_1$ by shifting  $a_1$ (resp. $a_2$) pendent edges from $v_1$ (resp. $v_2$) to $v_3$. We have
\begin{eqnarray*}
&  &  Mo_e(G_2 )-  Mo_e(G_1 ) =\\
 &  & (a_1+a_2+a_3+1-3)-(a_2+a_3+1-a_1-1)+(a_1+a_2+a_3+1-3)\\
                       &- & (a_2+3-a_1-a_3-1)+(a_1+a_2+a_3+3-1)-(a_1+a_3+3-a_2-1)\\
                       &+ & (a_1+a_2+a_3+3-1)-(a_2+a_3+3-a_1-1)+(a_1+a_2+a_3+3-1)\\
                       &- & (a_1+a_3+3-a_2-1)+(a_1+a_2+a_3+3-1)-(a_2+a_3+3-a_1-1)\\
                       &+ & (a_1+a_2+a_3+3-1)-(a_1+a_3+3-a_2-1)+(a_1+a_2+a_3+3-1)\\
                       &- & (a_2+a_3+3-a_1-1)\\
                       & = & 10a_1+6a_2+2a_3-8> 0.
\end{eqnarray*}
Thus, $ Mo_e(G )< Mo_e(G_1 )< Mo_e(G_2 )$. Clearly, $G_2 \cong H_3$, and by simple calculation, we obtain $Mo_e(H_3) =  m^2-m-48$.
\end{proof}

\begin{lemma}\label{38} Let $G \in \mathcal{G}_m^3$ with brace $\alpha_3 (1,2,2,3)$. Then
$Mo_e(G) < Mo_e(H_4)= m^2-3m-24$.
\end{lemma}
\begin{proof}
Suppose that $v_i$ $(i=1,...,6) $ be the six vertices in $\alpha_3 (1,2,2,3)$ of $G$ with $d_G(v_1)=d_G(v_2)=4$ and  $d_G(v_3)=d_G(v_4)=d_G(v_5)=d_G(v_6)=2$, as shown in Fig. \ref{alp2}.  Let $a_i$ be the number of pendent edges of $v_i$ $(i=1,...,6) $. Assume that $a_3 \geq  a_2$, and  $a_4+a_5+a_6 >1 $. Let $G_1$ be the graph obtained from $G$ by shifting $a_i$ ( $i \geq 4$) pendent edges from $v_i$ ( $i \geq 4$) to $v_3$. We get
\begin{eqnarray*}
 & & Mo_e(G_1 )-  Mo_e(G )=\\
  &  & (a_1+4-a_3- a_4-a_5+-a_6-1)-(a_1+ a_4+a_5+4-a_3-1)\\
                       &+ & (a_3+a_4+a_5+a_6+1-a_2-4)-(a_2+a_4+a_6+4-a_3-1)\\
                       & + & (a_1+a_3+a_4+a_5+a_6+4-1)-(a_1+a_3+a_5+4-a_4-1)\\
                       & + &(a_2+a_3+a_4+a_5+a_6+4-1)- (a_2+a_3+a_6+4-a_4-1)\\
                       & + & (a_1+a_2+a_3+a_4+a_5+a_6+5-1)-(a_1+a_2+a_2+a_4+5-a_5-a_6-1)\\
                       & + & (a_1+a_2+a_3+a_4+a_5+a_6+5-1)-(a_1+a_2+a_2+a_4+5-a_5-a_6-1)\\
                       &+ & (a_1+3-a_2-3)-(a_1+a_5+3-a_2-a_6-3)+(a_1+3-a_2-3)\\
                       & - & (a_1+a_5+3-a_2-a_6-3)\\
                       & = & 2(a_3+a_4+a_5)+6a_6-2a_2-12 > 0.
\end{eqnarray*}
For $a_1 +a_2 > 0$, let $G_2$ be the graph obtained from $G_1$ by shifting  $a_1$ (resp. $a_2$) pendent edges from $v_1$ (resp. $v_2$) to $v_3$. We have

\begin{eqnarray*}
 & & Mo_e(G_2 )-  Mo_e(G_1 ) =\\
 &  & (a_1+a_2+a_3+1-4)-(a_3+1-a_1-4)+(a_1+a_2+a_3+1-4)-(a_3+1-a_2-4)\\
                       &\qquad + & (a_1+a_2+a_3+4-1)-(a_1+a_3+4-1)+(a_1+a_2+a_3+4-1)-(a_2+a_3+4-1)\\
                       &\qquad + & (3-3)-(a_1+3-a_2-3)+ (3-3)-(a_1+3-a_2-3)\\
                       & = & 2a_1+6a_2> 0.
\end{eqnarray*}
Thus, $ Mo_e(G )< Mo_e(G_1 )< Mo_e(G_2 )$. Clearly, $G_2 \cong H_4$, and by simple calculation, we get $Mo_e(H_4) =  m^2-3m-248$.
\end{proof}

 \begin{figure}
\centering
\includegraphics[height=5cm, width=13cm]{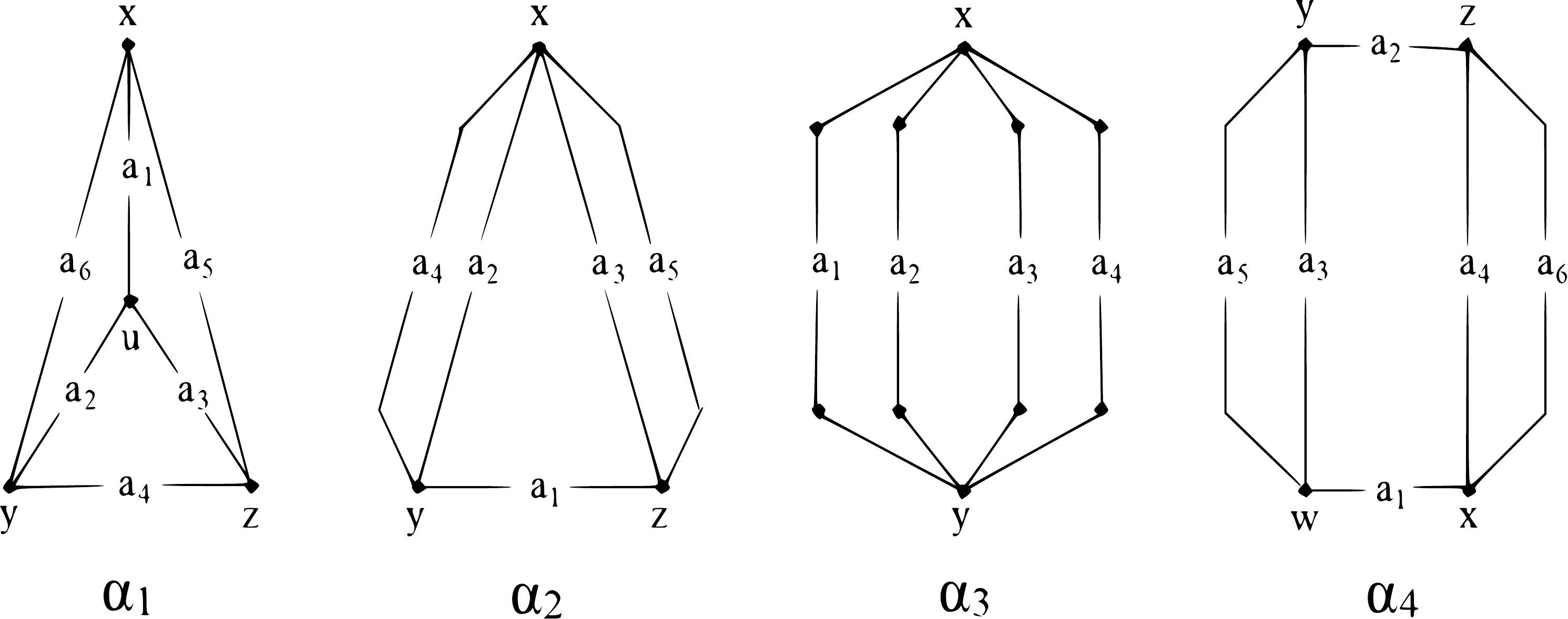}
\caption{ The Graphs for Lemmas \ref{41}, \ref{42}, \ref{43} and \ref{44}.}
\label{alp3}
\end{figure}

\begin{lemma}\label{43} Let $G \in \mathcal{G}_m^3$ of size $m$. Then
$Mo_e(G) < m^2-m-36$ for $m \geq 9$, and $Mo_e(G) \leq Mo_e(H_1)$ for $m \leq 9$.
\end{lemma}
\begin{proof}
Suppose that $G \in \mathcal{G}_m^3$, then $G$ has a brace $\alpha_2 (a_1, a_2, a_3,  a_4)$ as shown in Fig. \ref{alp3}. Assume that $1 \leq a_1 \leq a_2 \leq a_3 \leq a_4$. We proceed with the following three possible cases.

\noindent {\bf Case 1.} $3 \leq a_1 \leq a_2 \leq a_3 \leq a_4$.

We choose twelve edges, eight edges in the paths $P(a_i)$ $(i=1,2,3,4)$ such that each one is incident to $x$ or $y$, four edges in the middle of $P(a_i)$ $(i=1,2,3,4)$. We deduce that
\[
Mo_e(G) \leq 8(m-8)+ 4(m-12)+(m-12)(m-1) < m^2-m-36.
\]

\noindent {\bf Case 2.} $  a_1 = 2$.

\noindent {\bf Subcase 2.1.} $3  \leq a_2 \leq a_3 \leq a_4$.

We choose eleven edges, eight edges in the paths $P(a_i)$ $(i=1,2,3,4)$ such that each one is incident to $x$ or $y$, three edges in the middle of $P(a_i)$ $(i=2,3,4)$. We deduce that
\[
Mo_e(G) \leq 6(m-7)+ 2(m-9)+3(m-11)+(m-11)(m-1) < m^2-m-36.
\]

\noindent {\bf Subcase 2.2.} $ a_2 = a_3 = a_4= 2$.

The Subcase follows from Lemma \ref{37}.

\noindent {\bf Case 3.} $  a_1 = 1$.

\noindent {\bf Subcase 3.1.} $3  \leq a_2 \leq a_3 \leq a_4$.

We choose ten edges, six edges in the paths $P(a_i)$ $(i=2,3,4)$ such that each one is incident to $x$ or $y$, three edges in the middle of $P(a_i)$ $(i=2,3,4)$, and one edge in $P(a_1)$ incident to $x$. It follows that
\[
Mo_e(G) \leq 6(m-4)+ 4(m-10)+(m-10)(m-1) < m^2-m-36.
\]

\noindent {\bf Subcase 3.2.} $a_2=2,  3  \leq a_3 \leq a_4$.

The proof is similar to the Subcase 3.1.

\noindent {\bf Subcase 3.3.} $a_2= a_3=2,  3  \leq a_4$.

If $a_4=3 $, then it follows from Lemma \ref{38}. If $a_4 \geq 4$, then we choose nine edges, four edges in the path $P(a_4)$ such that two are incident to $x$ or $y$ and the other two are in the middle of  $P(a_4)$, two edges in the path $P(a_3)$ (resp. $P(a_2)$) such that one is incident to $x$ and the other is in the middle of $P(a_3)$ (resp. $P(a_2)$) and one edge in $P(a_1)$ incident to $x$.  We have
\[
Mo_e(G) \leq 2(m-5)+2(m-7)+ 4(m-6)+(m-9)+(m-9)(m-1) < m^2-m-36.
\]

\noindent {\bf Subcase 3.4.} $a_2= a_3=a_4=2$.

By Lemma \ref{36}, $Mo_e(G) < m^2-m-36$ for $m \geq 9$, and $Mo_e(G) \leq Mo_e(H_1)$ for $m \leq 9$.
\end{proof}

\begin{lemma}\label{44} Let $G \in \mathcal{G}_m^4$ of size $m$. Then $Mo_e(G) < m^2-m-36$.
\end{lemma}
\begin{proof}
Suppose that $G \in \mathcal{G}_m^4$, then $G$ has a brace $\alpha_4 (a_1, a_2, a_3,  a_4, a_5, a_6)$ as shown in Fig. \ref{alp3}. We choose eight edges, two edges in the path $P(a_5)$ such that each  is incident to $w$ or $y$, two edges in the path $P(a_6)$ such that each  is incident to $z$ or $x$, the four edges $yz, yw, wx, zx$.  We obtain
\[
Mo_e(G) \leq 4(m-5)+4(m-8)+ (m-8)(m-1) < m^2-m-36.
\]
\end{proof}

The proof of the Theorem \ref{a} follows from Lemmas \ref{31}, \ref{41}, \ref{42}, \ref{43} and \ref{44}.

\vspace{10mm}

\noindent {\bf Acknowledgement:} This work was supported  by the National Natural Science Foundation of China (Grant Nos. 12071194, 11571155 and 12071158).

\end{document}